\documentclass[11pt]{article}
\usepackage{amsthm,amssymb,amsmath,amscd}
\usepackage[all]{xy}
\usepackage{mathrsfs}
\usepackage{multirow}
\usepackage{color}
\usepackage{mathptm}
\usepackage{url}

\usepackage{caption}
\usepackage{subfigure}
\usepackage{graphicx}
\usepackage{tikz-cd}\usepackage{enumitem}
\DeclareMathAlphabet{\mathcal}{OMS}{cmsy}{m}{n}
\DeclareSymbolFont{largesymbols}{OMX}{cmex}{m}{n}

\newtheorem{Def}{Definition}[section]
\newtheorem{Prop}[Def]{Proposition}
\newtheorem{Theo}[Def]{Theorem}
\newtheorem{Lem}[Def]{Lemma}
\newtheorem{Koro}[Def]{Corollary}

\newcommand{\add}{\operatorname{add}\,}

\newcommand{\rad}{\operatorname{rad}\,}

\newcommand{\opp}{^{\rm op}}
\newcommand{\Imf}{{\rm Im}}

\newcommand{\gldim}{\operatorname{gl.dim}\,}

\newcommand{\pd}{{\rm pd }}
\newcommand{\id}{{\rm id}}

\newcommand{\Hom}{\operatorname{Hom}}
\newcommand{\rHom}{{\rm\bf R}{\rm Hom}}

\newcommand{\otimesL}{\otimes^{\rm\bf L}}

\newcommand{\Ext}{\operatorname{Ext}}
\newcommand{\Tor}{\operatorname{Tor}}

\newcommand{\D}[1]{\mathscr{D}(#1)}

\newcommand{\Db}[1]{ \mathscr{D}^{\rm b}(#1)}
\newcommand{\Dsg}[1]{ \mathscr{D}_{\rm sg}(#1)}

\newcommand{\C}[1]{\mathscr{C}(#1)}

\newcommand{\Kb}[1]{ \mathscr{K}^{\rm b}(#1)}
\newcommand{\modcat}[1]{#1\mbox{{\rm -mod}}}

\newcommand{\Modcat}[1]{#1\mbox{{\rm -Mod}}}
\newcommand{\stmodcat}[1]{#1\mbox{{\rm -{\underline{mod}}}}}

\newcommand{\pmodcat}[1]{#1\mbox{{\rm -proj}}}

\newcommand{\imodcat}[1]{#1\mbox{{\rm -inj}}}
\newcommand{\HH}{\rm HH}
\newcommand{\Imodcat}[1]{#1\mbox{{\rm -Inj}}}

\newcommand{\Pmodcat}[1]{#1\mbox{{\rm -Proj}}}
\newcommand{\lra}{\longrightarrow}

\newcommand{\ra}{\rightarrow}

\title{A characteristic free approach to skew-gentle algebras}
\author{Yiping Chen }
\date{}

\begin{document}
\maketitle

\renewcommand{\thefootnote}{\alph{footnote}}
\setcounter{footnote}{-1} 
\renewcommand{\thefootnote}{\alph{footnote}}
\setcounter{footnote}{-1} 

\begin{abstract}
To each skew-gentle algebra, one can assign a gentle algebra in terms of combinatorial data. In order to relate the structures of the two algebras, we establish a homological epimorphism and a recollement of derived module categories. This approach is characteristic free and works in particular also in characteristic two, which is the difficult case for skew-gentle algebras. This allows to solve open problems and to uniformly reprove and strengthen known results, for instance, (1) a complete classification of selfinjective skew-gentle algebras; (2) the finitistic dimension conjecture, Auslander and Reiten's conjecture, and Keller's conjecture hold for all skew-gentle algebras; (3) a precise connection of K-theory between skew-gentle algebras and gentle algebras; (4) all skew-gentle algebras are Gorenstein, and skew-gentle algebras and their gentle algebras share the same singularity categories. 
 \end{abstract}


\section{Introduction}

Gentle algebras were first introduced in \cite{AS1987} to classify the iterated tilted algebras of type $\widetilde{\mathbb{A}}$. Gentle algebras are Gorenstein \cite{Geis2005}, and stable under derived equivalences \cite{SZ2003}. Moreover, the derived categories and the singularity categories can be described explicitly \cite{Bekkert2003, Kalck2015}.
Recently, 
gentle algebras have become increasingly important due to their close connection with Fukaya categories \cite{HKK2017}.
Skew-gentle algebras were introduced in \cite{GP1999} to study a class of tame matrix problems, and proved to be closely related with gentle algebras.

Let $k$ be an algebraically closed field, and $A$ be a finite dimensional $k$-algebra. We write $char(k)$ for the characteristic of $k$. If $char(k)$ is not $2$, then a skew-gentle algebra $A$ is Morita equivalent to a skew group algebra $\Gamma G$ with $\Gamma$ a gentle algebra and $G$ a finite group with order $2$ \cite{GP1999}.
For a skew group algebra $\Gamma G$, when the order of $G$ is invertible in $k$, skew group constructions preserve many invariants such as Gorensteinness, selfinjectivity, finitistic dimensions (\cite[Theorem 1.3]{Reiten1985b}, \cite[Proposition 3.6]{huang2012}). Hence, gentle algebras and skew-gentle algebras share some common properties when $char(k)$ is not $2$. However, for $char(k)=2$, the situation is different and less is known.

The article is devoted to identifying uniform structure of skew-gentle algebras that allows a characteristic free approach. 
Throughout this paper, we assume that $k$ is an arbitrary algebraically closed field. 
A skew-gentle algebra $A(Q, I, Sp)$ is given by 
by combinatorial data $(Q, I, Sp)$, where $Q$ is a quiver, $I$ a set of relations and $Sp$ a set of loops (Section 3). $Q$ and $I$ define a gentle algebra $A(Q, I)$.  In this paper, we compare $A(Q, I, Sp)$ with $A(Q, I)$. Our main results are stated as follows.

\begin{Theo}\label{main result} Let $A=A(Q, I, Sp)$ be a skew-gentle algebra with $Sp\neq\emptyset$. Then, 

\begin{enumerate}[label=(\alph*)]
\item
We have a recollement 
\begin{equation}\label{recollement}\xymatrix{
\D{\Modcat{A/AeA}}\ar[rr]^{i_*=i_!}&&\D{\Modcat{A}}\ar@/^1pc/[ll]^{i^!}\ar@/_1pc/[ll]_{i^*}\ar[rr]^{j^!=j^*}&&\D{\Modcat{eAe}}\ar@/_1pc/[ll]_{j_!}\ar@/^1pc/[ll]^{j_*}},\end{equation}
where $$i^*=A/AeA\otimesL_A-, \; i_*=A/AeA\otimesL_{A/AeA}-, i^!=\rHom_A(A/AeA, -),$$
$$j_!=Ae\otimesL_{eAe}-, \;j^!=eA\otimesL_A-,\; j_*=\rHom_{eAe}(eA, -).$$
\item
The algebra $A/AeA$ is isomorphic to the gentle algebra $A(Q, I)$. 
\item 
The indecomposable direct factors of $eAe$ are hereditary algebras of linearly oriented quivers of type $\mathbb{A}_n$ for $n\geqslant 1$.

\end{enumerate}

\end{Theo}

The strategy to prove Theorem \ref{main result} is to find a gentle algebra $B$ and a surjective ring homomorphism $\varphi:{}A\ra B$ such that 
\begin{itemize}
\item
The algebra $B$ is given, up to isomorphism, by combinatorial data $(Q, I)$ in the definition of $A$. Thus, $B$ is a gentle algebra.  
In fact, the algebra $A(Q, I, Sp)$ is a deformation of $B$ \cite[Lemma 3.4]{Amiot2022}. When the set $Sp$ is empty, the algebra $A$ coincides with $B$. 
\item
The homomorphism $\varphi$ is a homological epimorphism, that is, $\D{\Modcat{B}}\ra \D{\Modcat{A}}$ is an embedding of triangulated categories.
\item
The kernel of $\varphi$ is a stratifying ideal of $A$. More precisely, there exists an idempotent $e$ of $A$ satisfying that $Ae\otimesL_{eAe}A\ra AeA$ is an isomorphism. The algebras $A/AeA$ and $B$ are isomorphic. 

\end{itemize}

It is important to note that the gentle algebras $B=A/AeA$ in Theorem \ref{main result} is different from the gentle algebra $\Gamma$ used in \cite{GP1999}, where for $char(k)\neq 2$, it is shown that $A(Q, I, Sp)$ is Morita equivalent to a skew group algebra $\Gamma G$. $B$ and $\Gamma$ have different properties; in \cite{ChenLu2015}, it is shown that $\Gamma$ and $B$ are not singularly equivalent in general. Thus Theorem \ref{main result} may provide a new information even when $char(k)\neq 2$.

\medskip

We get two kinds of consequences directly from these recollements. The first kind of results are about the relation between $A$ and its gentle algebra $B$, which in the following also is denoted by $A(Q, I)$. 
Using the theorem as a main tool, we clarify the structural connections between the skew-gentle algebra $A(Q, I, Sp)$ and the gentle algebra $A(Q, I)$ in a uniform way. 
We first reobtain that $A$ and $B$ are singularly equivalent, and all skew-gentle algebras are Gorenstein in a uniform way which has actually been proved in \cite{ChenLu2015} and \cite{ChenLu2017} by many detailed computations. A new result strengthens $(a)$ by showing that $A$ and $B$ are singularly equivalent of Morita type with level in the sense of Wang \cite{Wang2015}. Note that a stratifying recollement of bounded derived categories of indecomposable selfinjective algebras must split \cite{CK2017}. Then,
the third new result is a complete classification of all selfinjective skew-gentle algebras. 
Another new result relates K-theories of skew-gentle algebras and gentle algebras in a precise way using Corollary 6.8 in \cite{AKLY2017} ( also see Corollary 1.3 in \cite{ChenXi2016a}).

\begin{Koro}\label{coro}
\begin{enumerate}[label=(\alph*)]
\item 
The singularity categories $\Dsg{A(Q, I)}$ and $\Dsg{A(Q, I, Sp)}$ are equivalent as triangulated categories. In particular, the global dimension of $A(Q, I, Sp)$ is finite if and only if so is $A(Q, I)$.
\item
The skew-gentle algebras $A(Q, I, Sp)$ and $A(Q, I)$ are singularly equivalent of Morita type with level. 
\item
An indecomposable skew-gentle algebra $A(Q, I, Sp)$ is selfinjective if and only if $Sp=\emptyset$ and the gentle algebra $A(Q, I)$ is selfinjective. Equivalently, $Sp=\emptyset$ and the gentle $A(Q, I)$ is simple or a selfinjective Nakayama algebra with radical square zero. 
 
\item
There are isomorphisms of K-groups $\mathbb{K}_i(A(Q, I, Sp))\cong \mathbb{K}_i(A(Q, I))\oplus (\bigoplus_{Sp}\mathbb{K}_i(k))$, where $\bigoplus_{Sp}\mathbb{K}_i(k)$ is the direct sum of as the number of the special loops in $Sp$ copies of the $i$-th K-group of $k$, for all $i\in\mathbb{N}$.

\end{enumerate}

\end{Koro}

\medskip

The second kind of results are about general properties of skew-gentle algebras. In this part, we prove that the finitistic dimension conjecture, Auslander and Reiten's conjecture and Keller's conjecture \cite{Keller2018} hold for all skew-gentle algebras. Actually, all Gorenstein algebras satisfy the finitistic dimension conjecture. Hence, the finitistic dimension conjecture also follows from Corollary \ref{coro2} $(a)$. However, Corollary \ref{coro2} $(c)$ and $(d)$ are completely new.

\begin{Koro}\label{coro2}

\begin{enumerate}[label=(\alph*)]
\item
Skew-gentle algebras are Gorenstein. 

\item
The finitistic dimension conjecture holds for all skew-gentle algebras.
\item
Auslander and Reiten's conjecture holds for all skew-gentle algebras. 
\item
Keller's conjecture holds for all skew-gentle algebras.

\end{enumerate}

\end{Koro}

Since recollements of Gorenstein algebras can be completed to a ladder of height infinity, we get Corollary \ref{coro2} $(c)$ by Theorem 1.2 in \cite{CHQW2023}.
As Keller's conjecture is invariant under singular equivalences of Morita type with level, and holds for all gentle algebras \cite{ChenLiWang2020}, we can extend this conjecture to all skew-gentle algebras as an application of the main results.

This paper is organized as follows: In Section 2, we recall some notation and definitions. In Section 3, some basic facts about gentle algebras and skew-gentle algebras will be given, including precise description by quiver and relation. And in the final Section 4, we will prove the main results and corollaries. We will construct a stratifying ideal for a skew-gentle algebra $A(Q, I, Sp)$, and prove that quotient algebra of $A$ modulo the stratifying ideal is $A(Q, I)$.

\medskip

\section{Notation and definitions}

In this section, we will fix some notation  and recall some basic definitions needed in our proofs.

Throughout this article,  $k$ denotes a fixed algebraically closed field. Let $A$ be a finite dimensional $k$-algebra. We write $A^e$ for the enveloping algebra 
$A\otimes_kA^{op}$. We denote by $\Modcat{A}$ the category
of left $A$-modules, and by $\modcat{A}$ the category of all finitely
generated left $A$-modules. The notation Mod-$A$ means the category of right
$A$-modules. $\Pmodcat{A}$ (resp. $\pmodcat{A}$) is the full subcategory of $\Modcat{A}$ (resp. $\modcat{A}$) with objects the (resp. finitely generated) projective left $A$-modules. Dually, $\Imodcat{A}$ (resp. $\imodcat{A}$) is the full subcategory of $\Modcat{A}$ (resp. $\modcat{A}$) with objects the (resp. finitely generated) injective left $A$-modules.
By $\add(M)$, we shall mean the full subcategory of
$A$-Mod, whose objects are direct summands of finite direct sums of
copies of $M$.

For two morphisms $f:{}X \ra Y$ and $g:{}Y \ra Z$ in
$\modcat{A}$, the composition of $f$ and $g$ is written as $fg$, which is
a morphism from $X$ to $Z$. However, the composition of functors is from right to left. Namely, for two functors $F:{}\mathcal{C}\ra \mathcal{D}$ and $G:{}\mathcal{D}\ra \mathcal{E}$, we write $GF$ for their composition instead of $FG$, where $\mathcal{C}, \mathcal{D}$ and $\mathcal{E}$ are categories.

 Let $A$ be a finite dimensional $k$-algebra. 
 We denote by $\C{A}$ the category of complexes of finitely generated left $A$-modules, and by $\D{\Modcat{A}}$ the derived category of $\Modcat{A}$. The notions $\Kb{A}$ and $\Db{A}$ are the bounded homotopy and the bounded derived category of $\modcat{A}$, respectively. 

\medskip

\subsection{Recollements of triangulated categories}

In \cite{Beuilinson1982}, recollements were introduced by Beilinson, Bernstein and Deligne to study perverse sheaves.  A {\em recollement} of triangulated $k$-categories is a diagram

$$\xymatrix{
\mathcal{T}'\ar[rr]^{i_*=i_!}&&\mathcal{T}\ar@/^1pc/[ll]^{i^!}\ar@/_1pc/[ll]_{i^*}\ar[rr]^{j^!=j^*}&&\mathcal{T}''\ar@/_1pc/[ll]_{j_!}\ar@/^1pc/[ll]^{j_*}}$$
of triangulated categories and triangle functors such that

$(1)$ $(i^*, i_*), (i_!, i^!), (j_!, j^!), (j^*, j_*)$ are adjoint pairs;

$(2)$ $i_*, j_*, j_!$ are full and faithful;

$(3)$ $i^!j_*=0, j^!i_!=0, i^*j_!=0$.

$(4)$ Let $M$ be an object in $\mathcal{T}$. Then there are canonical triangles

$$i_!i^!(M)\lra M\lra j_*j^*(M)\lra i_!i^!(M)[1]$$

$$j_!j^!(M)\lra M\lra i_*i^*(M)\lra j_!j^!(M)[1]$$
where the morphisms are given by adjunctions.

To construct recollements, homological epimorphisms can be used.
Recall that
a ring homomorphism $\varphi:{}A\lra B$ is a {\em homological epimorphism} if $\varphi$ is a ring epimorphism and $\Tor^A_i(B, B)=0$ for all $i>0$. 
By \cite[Theorem 4.4]{GL1991} and \cite[Section 4]{NicolasSaorin2009a}, $\varphi$ is a homological epimorphism if and only if the induced map $\Ext^i_B(M, N)\lra \Ext^i_A(\varphi^*(M), \varphi^*(N))$ is bijective for all $M, N$ and $i\geqslant 0$, equivalently, 
the induced functor $\varphi_*:{}\D{\Modcat{B}}\lra \D{\Modcat{A}}$ is fully faithful. Let $q:{} \D{\Modcat{A}}\lra \D{\Modcat{A}}/\Imf(\varphi_*)$ be the Verdier quotient of $\D{\Modcat{A}}$ modulo the essential image of $\varphi_*$. Note that $\varphi_*$ has a left adjoint $B\otimesL_A-$ and a right adjoint $\rHom_A(B, -)$. Then, by \cite[Proposition 9.1.18]{Neeman2001}, the Bousfield (co)localisation functor exists. Namely, the Verdier quotient $q$ has a left adjoint and a right adjoint. 
In this way, we get
a recollement
$$\xymatrix{
\D{\Modcat{B}}\ar[rr]^{\varphi_*}&&\D{\Modcat{A}}\ar@/^1pc/[ll]^{\rHom_A(B, -)}\ar@/_1pc/[ll]_{B\otimesL_A-}\ar[rr]^{}&&\chi\ar@/_1pc/[ll]\ar@/^1pc/[ll]}$$
where $\chi$ is a triangulated category.

Let $\varphi:{}A\lra B$ be a surjective homomorphism. Assume that $\varphi$ is a homological epimorphism with the kernel $AeA$ where $e$ is an idempotent of $A$. By \cite[Remark 2.1.2 and Definition 2.1.1]{Cline1996}, the kernel $AeA$ is a {\em stratifying ideal}, i.e., $Ae\otimesL_{eAe}eA\lra AeA$ is an isomorphism. By \cite[Example 4.5]{AKL2011}, this gives a recollement
$$\xymatrix{
\D{\Modcat{A/AeA}}\ar[rr]^{i_*}&&\D{\Modcat{A}}\ar@/^1pc/[ll]^{i^!}\ar@/_1pc/[ll]_{i^*}\ar[rr]^{j^!}&&\D{\Modcat{eAe}}\ar@/_1pc/[ll]_{j_!}\ar@/^1pc/[ll]^{j_*}}$$
where
$$i^*=A/AeA\otimesL_A-, \; i_*=A/AeA\otimesL_{A/AeA}-, i^!=\rHom_A(A/AeA, -),$$
$$j_!=Ae\otimesL_{eAe}-, \;j^!=eA\otimesL_A-,\; j_*=\rHom_{eAe}(eA, -).$$

\medskip

 Let $A$ be an algebra, and $M$ be a finitely generated left $A$-module. The finitely generated module $M$ is called {\em Gorenstein projective} if there is an exact sequence $P^{\bullet}\in \C{A}:$
$$\cdots\lra P^{-2}\lra P^{-1}\stackrel{d^{-1}}\lra P^0\stackrel{d^0}\lra P^1\lra P^2\lra \cdots$$
such that $\Hom_A(P^{\bullet}, A)$ is exact and $X\cong \Imf(d^0)$. Clearly, all finitely generated projective left $A$-modules are finitely generated Gorenstein projective. For any $M\in\modcat{A}$, we write $\id_A(M)$ for the injective dimension of $M$.
An algebra $A$ is called {\em Gorenstein} if $\id_A(A)<\infty$ and $\id_{A^{\opp}}(A)<\infty$. We write $d$ for the injective dimension of the left regular module $A$. Then, the algebra $A$ is Gorenstein if and only if $\Omega^d(M)$ is Gorenstein-projective for all $M\in \modcat{A}$, where $\Omega$ is the syzygy functor of the stable category $\stmodcat{A}$ \cite[Proposition 3.10]{Beligiannis2005}. 
It is also equivalent to saying $\Kb{\pmodcat{A}}=\Kb{\imodcat{A}}$ \cite[Section 1]{Happel1991f}.

For convenience, we write $(\D{\Modcat{B}}, \D{\Modcat{A}}, \D{\Modcat{C}})$ for the recollement
\begin{equation}\label{R}
\xymatrix{
\D{\Modcat{B}}\ar[rr]^{i_*}&&\D{\Modcat{A}}\ar@/^1pc/[ll]^{i^!}\ar@/_1pc/[ll]_{i^*}\ar[rr]^{j^!}&&\D{\Modcat{C}}.\ar@/_1pc/[ll]_{j_!}\ar@/^1pc/[ll]^{j_*}}
\end{equation}
We say that the recollement (\ref{R}) can be restricted to $\Db{\rm mod}$ and $\Kb{\rm proj}$ if $(\Db{B}, \Db{A}, \Db{C})$ and $(\Kb{B}, \Kb{A}, \Kb{C})$ are recollements with the same triangle functors in (\ref{R}). 
The following Propositions, taken from \cite{AKLY2017}, \cite{CK2017} and \cite{QinHan2016}, relate Gorensteinness and recollements.

\begin{Prop}\label{1} (\cite[Corollary 4.7]{AKLY2017}) If $\gldim{C}<\infty$, then the recollement (\ref{R}) can be restricted to $\Db{\rm mod}$.
\end{Prop}

\begin{Prop}\label{2} (\cite[Proposition 3.4]{CK2017}, \cite{QinHan2016})
Assume the recollement (\ref{R}) can be restricted to $\Db{\rm mod}$. Then, the following statements are equivalent:
\begin{enumerate}[label=(\alph*)]
\item
The algebra $A$ is Gorenstein.
\item
The algebras $A$, $B$ and $C$ are Gorenstein.
\item
The algebras $B$ and $C$ are Gorenstein, and the recollement (\ref{R}) can be restricted to $\Kb{{\rm proj}}$.
\item
The algebras $B$ and $C$ are Gorenstein, and $i^!(A)\in\Kb{\pmodcat{B}}$.
\end{enumerate}
\end{Prop}

\subsection{Singular equivalences of Morita type (with level)}

Let $A$ and $B$ be finite dimensional $k$-algebras. We say that $A$ and $B$ are {\em derived equivalent} if $\Db{A}$ and $\Db{B}$ are equivalent as triangulated categories. Immediately, derived equivalences are a special kind of recollements, where one outer algebra is zero. 
We write $\Dsg{A}$ for the singularity category of $A$, that is $\Dsg{A}$ is the Verdier quotient of $\Db{A}$ by the full subcategory $\Kb{\pmodcat{A}}$. The global dimension of $A$ is finite if and only if $\Dsg{A}$ is trivial. 
Two algebras $A$ and $B$ are said to be {\em singularly equivalent} if $\Dsg{A}$ and $\Dsg{B}$ are equivalent as triangulated categories.

 In \cite{Chen}, Chen and Sun extend the classical stable equivalences of Morita type, and introduce singular equivalences of Morita type. It is well-known that stable equivalences of Morita type induce stable equivalences. As a generalization, Chen and Sun prove that singular equivalences of Morita type give rise to singular equivalences. Later, the notion has been generalized by Wang in \cite{Wang2015} to singular equivalence of Morita type with level. 
Namely, there are bimodules ${}_BM_A$ and ${}_AN_B$ which are projective on each side, satisfying 
$$N\otimes_BM\cong \Omega^l_{A^e}(A) \mbox{ in }\stmodcat{A^e}\;\;\mbox{   and  }\;\;M\otimes_AN\cong \Omega^l_{B^e}(B)\mbox{ in }\stmodcat{B^e}$$
for $l\in\mathbb{N}$, where $\Omega$ is the syzygy endofunctor of the stable category. 
In particular, $A$ and $B$ are singularly equivalent of Morita type for $l=0$. Wang proved that singular equivalences of Morita type with level also supply singular equivalences, and all derived equivalences of rings are special kinds of singular equivalences of Morita type with level. 

Conversely, singular equivalences imply singular equivalences of Morita type with level under certain conditions. 
Assume that both $A$ and $B$ are Gorenstein. Let $X^{\bullet}$ be a complex of finitely generated $B$-$A^{\opp}$-bimodules such that $X^{\bullet}$ is isomorphic to some object in $\Kb{\pmodcat{B}}$ (resp. $\Kb{\pmodcat{A^{\opp}}}$). If the functor $X^{\bullet}\otimesL_A-:{}\Db{A}\ra \Db{B}$ induces a singular equivalence, then it induces a singular equivalences of Morita type with level \cite[Theorem]{Dalezios2021}.

\subsection{Homological conjectures}

The famous finitistic dimension conjecture and Auslander and Reiten's conjecture, which are still open now, play important roles in representation theory of algebra and homological algebra.

{\bf The finitistic dimension conjecture}: {\em The finitistic dimension of $A$ (the supremum of the projective dimensions of finitely generated $A$-modules with finite projective dimension) is finite. }

 {\bf Auslander and Reiten's Conjecture}: {\em A finitely generated $A$-module $X$ satisfying $\Ext_A^i(X, X\oplus A)=0$ for all $i>0$ is necessarily projective.}

Auslander and Reiten's conjecture has been first introduced in \cite{Auslander1975c} as an analogue of M\"uller's theorem. Auslander and Reiten's conjecture
holds for all algebras if and only if generalized Nakayama conjecture holds for all algebras. If an algebra satisfies generalized Nakayama conjecture, then it satisfies the celebrated Nakayama conjecture \cite{Auslander1975c}. 

Auslander and Reiten's conjecture is closely related with the finitistic dimension conjecture.
 More precisely, when the finitistic dimension conjecture holds for all algebras, so does Auslander and Reiten's conjecture. It is still unknown whether Auslander and Reiten's conjecture holds automatically for an algebra satisfying the finitistic dimension conjecture \cite[Theorem 3.4.3]{Yamagata1996}. It is also unknown if or not Auslander and Reiten's conjecture holds for all Gorenstein algebras. Auslander and Reiten's conjecture for self-injective algebras is known as Tachikawa's second conjecture, which is still open now.

The finitistic dimension conjecture holds for some classes of algebras such as monomial algebras, algebras with radical cube zero, and algebras whose representation dimension is no more than $3$ \cite{Green1991, Green1991a, Igusa2005}. For the history of this conjecture, we refer the readers to \cite{Zimmermann1995}.
Auslander and Reiten's conjecture holds for all syzygy-finite algebras. These include algebras with finite global dimensions, algebras of finite representation type, algebras with radical square zero and monomial algebras \cite{Auslander1975c}. And it is true for
symmetric biserial algebras, and for local algebras with radical cube zero \cite{Xu2013a, Xu2015}. It also holds for certain commutative rings that are not necessarily finite dimensional algebras, such as 
 Gorenstein commutative local rings with radical cube zero, complete intersection local rings and some commutative CM rings \cite{Auslander1993, Hoshino1984, Huneke2004a}.

\medskip

\subsection{Keller's conjecture}

Keller's conjecture is about the higher structure of Hochschild cohomology of dg categories. Let $\Lambda$ be a finite dimensional $k$-algebra. We write $\HH^*_{sg}(\Lambda, \Lambda)$ for the singular Hochschild cohomology of $\Lambda$ which is given by 
$$\HH^n_{sg}(\Lambda, \Lambda):=\Hom_{\Dsg{\Lambda^e}}(\Lambda, \Lambda[n])$$
for all $n\in \mathbb{Z}$, where $[1]$ is the suspension functor of the singularity category $\Dsg{\Lambda^e}$.

In \cite{LV2005}, Lowen and Van den Bergh prove that the Hochschild cohomology of $\Lambda$ and the Hochschild cohomology of the dg enhancement of the derived category of $\Lambda$ are isomorphic as graded algebras, and that isomorphism lifts to $B_{\infty}$-level. Namely, the Hochschild cohomology cochain complex $C^*(\Lambda, \Lambda)$ and the Hochschild cohomology cochain complex of the dg enhancement of the derived category are isomorphic as $B_\infty$-algebras.

We denote the dg enhancement of the singularity category $\Dsg{\Lambda}$ by ${\bf S}_{dg}(\Lambda)$.
Note that both the singular Hochschild cohomology complex 
$C^*_{sg}(\Lambda, \Lambda)$ and the Hochschild cochain complex $C^*({\bf S}_{dg}(\Lambda), {\bf S}_{dg}(\Lambda))$ are endowed with $B_\infty$-structures \cite{Ezra1994, Wang2021}. 
 In \cite{Keller2018}, Keller proposes a question about an anologue of Lowen and Van den Bergh's theorem for singularity categories. He proves that the singular Hochschild cohomology of $\Lambda$ and the Hochschild cohomology of the dg singularity category are isomorphic as graded algebras, and conjectures that the isomorphism also lifts to $B_{\infty}$-level.

 {\bf Keller's conjecture}: There is an isomorphism in ${\rm Ho}(B_\infty)$
$$C^*_{sg}(\Lambda, \Lambda)\cong C^*({\bf S}_{dg}(\Lambda), {\bf S}_{dg}(\Lambda)),$$ 
where ${\rm Ho}(B_{\infty})$ is the homotopy category of $B_\infty$-algebras.

In \cite{ChenLiWang2020}, Chen, Li and Wang show that Keller's conjecture is invariant under singular equivalences of Morita type with level, and holds for all gentle algebras.

\medskip

\section{Skew-gentle algebras}

Recall that a {\em quiver} $Q=(Q_0, Q_1, s, t)$ is an oriented graph, where $Q_0$ is the set of vertices and $Q_1$ is the set of arrows. The maps $s:{}Q_1\lra Q_0$ and $t:{}Q_1\lra Q_0$ are defined by $s(\alpha)=i$ and $t(\alpha)=j$ for any arrow $\alpha:{}i\lra j$. Namely, $s(\alpha)$ is the source of $\alpha$, and $t(\alpha)$ is the terminal of $\alpha$. 
A sub-quiver of $Q$ is a quiver $(Q_0', Q_1', s', t')$ satisfying that $Q'_0\subseteq Q_0, Q'_1\subseteq Q_1$, and $s', t'$ are the restrictions of $s$ and $t$ to $Q'_1$, respectively. Throughout this article, all quivers are supposed to be finite, that is, both $Q_0$ and $Q_1$ are finite sets.

Let $Q=(Q_0, Q_1, s, t)$ be a quiver. For any $l\in\mathbb{N}$, a path of length $l$ is an oriented sequence of arrows $p=\alpha_1\alpha_2\cdots\alpha_l$ such that $t(\alpha_i)=s(\alpha_{i+1})$ for all $i=1, \cdots, l-1$. When $l=0$, the path is called a {\em trivial path}, and is denoted by $e_i$ for $i\in Q_0$. A nontrivial path $p$ is an {\em oriented cycle} if $s(p)=e(p)$. An oriented cycle of length $1$ is also said to be a {\em loop}.

A quiver algebra $kQ$ is a $k$-algebra whose underlying $k$-vector space has as a basis the set of all paths in $Q$. Let $p$ and $p'$ be two paths in $Q$. The product of $p$ and $p'$ is defined by $\delta_{t(p)s(p')}pp'$, where $\delta$ is the Kronecker delta.

A relation $\sigma$ is a $k$-linear combination of paths of the form $\sigma=k_1p_1+k_2p_2+\cdots +k_np_n$ with $k_i\in k$ and $s(p_i)=s(p_j), e(p_i)=e(p_j)$ for all $1\leqslant i, j\leqslant n$. Let $\rho$ be the set of relations. The pair $(Q, \rho)$ is called {\em a quiver with relations}. We write $kQ/\langle \rho\rangle$ for the quotient algebra of $kQ$ modulo the ideal generated by the set $\rho$ in $kQ$.  Let $J$ be the ideal generated by all arrows in $kQ$. The ideal $\langle \rho\rangle$ is called {\em admissible} if $J^t\subseteq \langle \rho\rangle \subseteq J^2$ for some integer $t\geqslant 2$. Let $Q$ be a finite quiver, and $\langle \rho\rangle$ be an admissible ideal of $kQ$. Then $kQ/\langle \rho\rangle$ is a finite dimensional algebra. Throughout this paper, we write $A(Q, \rho)$ for the bound quiver algebra corresponding to the pair $(Q, \rho)$.
For more facts about quiver algebras, we refer the readers to \cite[Chapter III]{Auslander1995} and \cite[Chapter II]{Assem2006}.

\medskip

A quiver with relations $(Q, I)$ is called a {\em gentle pair}, if the following hold.
\begin{itemize}
\item
Each vertex of $Q$ is start point of at most two arrows, and end point of at most two arrows.
\item
For each arrow $\alpha$ in $Q$, there is at most one arrow $\beta$ with $t(\alpha)=s(\beta)$ such that $\alpha\beta\notin I$, and at most one arrow $\gamma$ with $t(\gamma)=s(\alpha)$ such that $\gamma\alpha\notin I$.
\item
For each arrow $\alpha$ in $Q$, there is at most one arrow $\beta$ with $t(\alpha)=s(\beta)$ such that $\alpha\beta\in I$, and there is at most one arrow $\gamma$ with $t(\gamma)=s(\alpha)$ such that $\gamma\alpha\in I$.
\item
The algebra $A(Q, I)$ is finite dimensional. 
\end{itemize}

A finite dimensional algebra $A(Q, I)$ is called a {\em gentle algebra} if $(Q, I)$ is a gentle pair. Let $(Q, I)$ be a gentle pair. We add some special loops in $Q$, and denote the set of special loops by Sp. We call $(Q, I, Sp)$ the {\em skew-gentle triple} if $(Q', I\cup\{\alpha^2\mid \alpha\in Sp\})$ is a gentle pair where $Q'$ is the quiver by adding the special loops to $Q$. The finite dimensional algebra $A(Q, I, Sp)=kQ'/\langle I\cup \{\alpha^2-\alpha\mid \alpha\in Sp\}\rangle$ is called a {\em skew-gentle algebra} corresponding to the skew-gentle triple $(Q, I, Sp)$. A gentle algebra is a skew-gentle algebra with Sp empty. Equivalently, $(Q, I)$ is a gentle pair for any skew-gentle triple $(Q, I, Sp)$.

The ideal $\langle I\cup \{\alpha^2-\alpha\mid \alpha\in Sp\}\rangle$ is not admissible. Thus, we cannot read off the quiver of a skew-gentle algebra by the corresponding skew-gentle triple $(Q, I, Sp)$. To obtain the quiver, we proceed as follows. 
For a skew-gentle triple $(Q, I, Sp)$, the vertices of $Q$ consist of the special vertices where a special loop is attached and the remaining ordinary vertices. We denote the sets of special vertices and ordinary vertices by $Q_0^{Sp}$ and $Q_0^{or}$, respectively. Then, we split every special vertex into two. More precisely, for a special vertex $i\in Q_0^{Sp}$, we define two vertices $i^-$ and $i^+$. This yields two sets of vertices:
$$Q_0^{Sp-}:=\{i^-\mid i\in Q_0^{Sp}\}; \quad Q_0^{Sp+}:=\{i^+\mid i\in Q_0^{Sp}\}.$$

The quiver of the skew-gentle algebra $A(Q, I, Sp)$ is described as follows:

\begin{itemize}
\item
The set of vertices $Q^A_0$: $Q_0^{or}\cup Q_0^{Sp+}\cup Q_0^{Sp-}$.
\item
The set of arrows $Q^A_1$: For an arrow $\alpha:{}i\lra j$ in $Q$
\begin{itemize}
\item
if both $i$ and $j$ are ordinary, then $\alpha$ is an arrow in $Q_1^A$.
\item
if $i\in Q_0^{Sp}$ and $j\in Q_0^{or}$, then there are two arrows 
$$
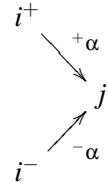
\xymatrix@=5mm{
i^+\ar[rd]^{{}^+\alpha}&\\&j\\i^-\ar[ur]_{{}^-\alpha}&
}$$ in $Q_1^A$.
\item
dually, when $i\in Q_0^{or}$ and $j\in Q_0^{Sp}$, we also get two arrows $\alpha^+:{}i\lra j^+$ and $\alpha^-:{}i\lra j^-$ in $Q_1^A$.
\item
if both $i$ and $j$ are in $Q_0^{Sp}$, we have four arrows in $Q_1^A$.
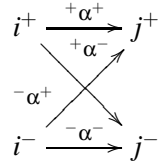
$$\xymatrix@=10mm{
i^+\ar[r]^{{}^{+}\alpha^+}\ar[rd]^(.3){{}^+\alpha^-}&j^+\\
i^-\ar[r]^{{}^-\alpha^-}\ar[ur]^(.3){{}^-\alpha^+}&j^-
}$$
\end{itemize}
\end{itemize}

Note that $I$ consists of some paths of length $2$. The set of relations $I^A$ is as follows. For a path
$$j\stackrel{\alpha}\lra i\stackrel{\beta}\lra k$$
in $I$,
\begin{itemize}
\item
The zero relations:{}
If $i$ is an ordinary vertex, then the paths ${}^{\pm}\alpha\beta^{\pm}, {}^\pm \alpha\beta$, $\alpha\beta^{\pm}$ and $\alpha\beta$ lie in $I^A$. The existence of ${}^{\pm}\alpha$ and $\beta^{\pm}$ depends on if or not $j$ and $k$ are special vertices. 
\item
The commutative relations:{}
If $i$ is a special vertex, then there are three potential cases where all squares are anti-commutative. 
$$
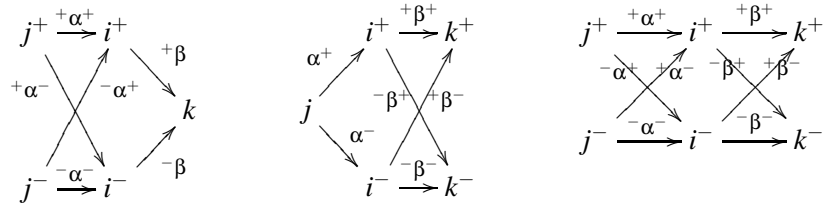
\xymatrix@=5mm{
j^+\ar[r]^{{}^+\alpha^+}\ar[rdd]_(.3){{}^+\alpha^-}&i^+\ar[rd]^{{}^+\beta}&\\&&k\\
j^-\ar[r]^{{}^-\alpha^-}\ar[uur]_(.7){{}^-\alpha^+}&i^-\ar[ur]_{{}^-\beta}&
}\quad\quad\quad\xymatrix@=5mm{
&i^+\ar[r]^{{}^+\beta^+}\ar[rdd]^{{}^+\beta^-}&k^+\\
j\ar[ur]^{\alpha^+}\ar[rd]^{\alpha^-}&&\\
&i^-\ar[r]^{{}^-\beta^-}\ar[uur]^{{}^-\beta^+}&k^-
}\quad\quad\quad \xymatrix{
j^+\ar[r]^{{}^+\alpha^+}\ar[rd]^{{}^+\alpha^-}&i^+\ar[r]^{{}^+\beta^+}\ar[rd]^{{}^+\beta^-}&k^+\\
j^-\ar[r]^{{}^-\alpha^-}\ar[ur]^{{}^-\alpha^+}&i^-\ar[r]^{{}^-\beta^-}\ar[ur]^{{}^-\beta^+}&k^-
}$$

\end{itemize}

\medskip

\section{Main results}

In this section, we will turn to the proof of the main results. The key point is to construct a stratifying ideal of $A(Q, I, Sp)$, using a well-known skill to construct quasi-hereditary algebras \cite[Section 4.5]{ParshallScott1987}.

Let $B$ and $C$ be rings with identity, and ${}_BM_C, {}_CN_B$ be bimodules. Recall that $(B, C, {}_BM_C, {}_CN_B, f, g)$ is a Morita context if $f:{}M\otimes_CN\lra B$ and $g:{}N\otimes_BM\lra C$ are morphisms of bimodules satisfying that 
$$f(m\otimes n)m'=mg(n\otimes m')\mbox{ and }nf(m\otimes n')=g(n\otimes m)n'$$
for all $m, m'\in M$ and $n, n'\in N$. With such Morita context, we associate the ring
$$A=\small{\begin{pmatrix}B&M\\N&C
\end{pmatrix}}
$$
with the formal operations of $2\times 2$-matrices, 
$$\small{\begin{pmatrix}b&m\\n&c
\end{pmatrix}\begin{pmatrix}b'&m'\\n'&c'
\end{pmatrix}:=\begin{pmatrix}bb'+f(m\otimes n')&bm'+mc'\\nc'+cn'&g(n\otimes m')+cc'
\end{pmatrix}}
$$
with $b, b'\in B, c, c'\in C, m, m'\in M, n, n'\in N$. The ring $A$ is called the {\em Morita context ring} associated with the Morita context $(B, C, {}_BM_C, {}_CN_B, f, g)$.
For a Morita context ring $A$, we suppose $g:{}N\otimes_AM\lra \rad(C)$. Then $f(M\otimes_C N)$ is a nilpotent ideal of $B$, because
$$f(m\otimes n)f(m'\otimes n')=f(m\otimes nf(m'\otimes n'))=f(m\otimes g(n\otimes m')\otimes n')$$
for all $m, m'\in M$ and $n, n'\in N$. Let$$\small{E_{22}=\begin{pmatrix}0&0\\0&e
\end{pmatrix}.}
$$ be an idempotent of $A$. In fact, $E_{22}$ is the idempotent $e$ in Theorem \ref{main result} by abuse of notation. 
By computation, the ideal
$$AE_{22}A=\begin{pmatrix}f(M\otimes_CN)&M\\N&C
\end{pmatrix}.$$
If $f(M\otimes_CN)=M\otimes_CN$, then $AE_{22}\otimesL_{C}E_{22}A$ is isomorphic to $AE_{22}A$. The following Lemma has been proved in \cite{Koenig2009}.

\begin{Lem}\cite[Lemma 4.1]{Koenig2009}\label{KN09} If $\Tor^C_n(M, N)=0$ for all $n\geqslant 1$, then $A E_{22}A$ is a stratifying ideal. 
\end{Lem}

In the following, we will show that every skew-gentle algebra with ${Sp}\neq\emptyset$ is isomorphic to a Morita context ring with respect to the Morita context $(B, C, M, N, f, g)$ such that $g:{}N\otimes_CM\lra \rad(B)$, and $M, N$ are projective right $C$-module and projective left $C$-module, respectively.

We write $e$ for the sum of idempotents (or trivial paths) corresponding to the vertices in $Q_0^{Sp-}$. Then, $1=e+(1-e)$ is a decomposition of the identity of $A(Q, I, Sp)$. The Peirce decomposition of $A(Q, I, Sp)$ will be written in the form
$$A(Q, I, Sp)=\begin{pmatrix}(1-e)A(Q, I, Sp)(1-e)&(1-e)A(Q, I, Sp)e\\eA(Q, I, Sp)(1-e)&eA(Q, I, Sp)e
\end{pmatrix}.
$$

In the following, we write $A$ for the skew-gentle algebra $A(Q, I, Sp)$, and write
$$A=\begin{pmatrix}B&M\\N&C
\end{pmatrix}
$$
for the above Peirce decomposition of $A$ for simplicity. 
Here, the sets $M$ and $N$ are left $B$-right $C$-bimodule and left $C$-right $B$-bimodule, respectively. As a $k$-vector space, the basis of $M$ (resp. $N$) consists of
$$\{p+\langle I^A\rangle\mid p \mbox{ is a path in } Q^A \mbox{ with } s(p)\in Q^{Sp+}_0\cup Q^{or}_0, e(p)\in Q^{Sp-}_0\}$$$$ (resp. \{p+\langle I^A\rangle\mid p \mbox{ is a path in } Q^A \mbox{ with } s(p)\in Q^{Sp-}, e(p)\in Q^{Sp+}_0\cup Q_0^{or}\}).$$
The left $B$-right $C$-bimodule homomorphism 
$f:{}M\otimes_{C}N\lra B$
and left $C$-right $B$-bimodule homomorphism
$g:{}N\otimes_{B}M\lra C$ are given by the multiplication in $A$. Due to the associativity of the multiplication in $A$, the bimodule homomorphisms are compatible with each other. In this way, $A$ is a Morita context ring associated with the Morita context $(B, C, M, N, f, g)$.

Now, we want to construct a stratifying ideal of $A$. 
First, we will explore some properties of skew-gentle triples.

\begin{Lem} \label{algebra C}Let $(Q, I, Sp)$ be a skew-gentle triple with $Sp\neq\emptyset$. Then, for any positive integer $l$ greater than $1$, given $i^-, j^-, k^-\in Q_0^{Sp-}$:
\begin{enumerate}[label=(\alph*)]
\item 
There is at most one vertex $j^-\in Q_0^{Sp-}$ such that $i^-\lra j^-$ is an arrow in $Q^A$. Dually, there is at most one vertex $k^-\in Q_0^{Sp-}$ such that $k^-\lra i^-$ is an arrow in $Q^A$.

\item
There is at most one vertex $k\in Q^{or}_0\cup Q^{Sp+}_0$ such that $k\lra j^-$ is an arrow in $Q^A$. Dually, there is at most one vertex $i\in Q^{or}_0\cup Q^{Sp+}_0$ such that $j^-\lra i$ is an arrow in $Q^A$.

\item
There is not an oriented cycle $p:{}i^-\lra \cdots\lra i^-$ in $Q^A$ such that $p\notin \langle I^A\rangle$.

\item
Let $p:{}i^-\stackrel{\alpha_1}\lra \cdots \stackrel{\alpha_l}\lra j^-$ be a path in $Q^A$ satisfying $p\notin \langle I^A\rangle$ and $t(\alpha_i)\in Q^{or}_0\cup Q^{Sp+}_0$ for all $i=1, \cdots, l-1$.
\begin{enumerate}[label=(d\arabic*)]

\item
If there exists some $t(\alpha_{i_0})\in Q^{Sp+}_0$ with $1\leqslant t_0\leqslant l-1$, then there is a path $p'\in Q^A$ through some vertex in $Q^{Sp-}_0$ such that $p+p'\in \langle I^A\rangle$.
\item
Assume that $t(\alpha_i)\in Q^{Sp+}_0$ for all $i=1, \cdots, l-1$. Then, there is a path $p':{}i^-\stackrel{\beta_1}\lra \cdots\stackrel{\beta_l}\lra j^-$ in $Q^A$ such that $p-(-1)^lp'\in\langle I^A\rangle$ and $t(\beta_i)\in Q^{Sp-}_0$ for all $i=1, \cdots, l-1$.

\item
Suppose that all of $\alpha_{i}$ satisfies that $t(\alpha_{i})\in Q^{or}_0$ for all $i=1, \cdots, l-1$. Then, the path $p$ satisfying $p\notin\langle I^A\rangle$ is unique in $Q^A$. Namely, if $p':{}i^-\stackrel{\beta_1}\lra \cdots\stackrel{\beta_l}\lra j^-$ is another path with $t(\beta_{j})\in Q^{or}_0$ for all $j=1, \cdots, l-1$ in $Q^A$, then $p'\in \langle I^A\rangle$.

\end{enumerate}

\item
Dually, let $p:{}i\stackrel{\alpha_1}\lra \cdots\stackrel{\alpha_l}\lra j$ be a path in $Q^A$ with $i, j\in Q^{or}_0\cup Q^{Sp+}_0$ and $t(\alpha_i)\in Q_0^{Sp-}$ for all $l=1, \cdots, l-1$. Then, there exists a path $p':{}i\stackrel{\beta_1}\lra \cdots\stackrel{\beta_l}\lra j$ in $Q^A$ with $t(\beta_i)\in Q^{Sp+}_0$ satisfying $p-(-1)^lp'\in\langle I^A\rangle$ for all $i=1, \cdots, l-1$.

\item

Let $p_1:{}i^-\stackrel{\alpha_1}\lra \cdots\stackrel{\alpha_l}\lra j^-$ and $p_2:{}j^-\stackrel{\beta_1}\lra \cdots\stackrel{\beta_q}\lra k^-$ be paths in $Q^A$ such that $p_1, p_2\notin \langle I^A\rangle$
with $t(\alpha_i), t(\beta_j)\in Q^{or}_0\cup Q^{Sp+}_0\cup \emptyset$ for all $i=1, \cdots, l-1, j=1, \cdots, q-1$. Then, $p_1p_2\notin \langle I^A\rangle$. Here, the notion $t(\alpha_{i_0})=\emptyset$ for some $i_0$ implies that the path $p_1$ is an arrow $i^-\lra j^-$.

\item 
Let $\alpha_1:{}j\lra i^-, \alpha_2:{}j\lra k^-, \beta_1:{}j^-\lra i$ and $\beta_2:{}k^-\lra i$ be arrows in $Q^A$ with $i, j\in Q^{or}_0\cup Q^{Sp+}_0$.
\begin{enumerate}[label=(g\arabic*)]
\item
For all paths $p_1\alpha_1p_2$ and $p_3\alpha_2p_4$ in $Q^A$ satisfying $p_1\alpha_1 p_2, p_3\alpha_2p_4\notin I^A$, we have 
$(p_1\alpha_1p_2+\langle I^A\rangle)+(p_3\alpha_2p_4+\langle I^A\rangle)=$
$$\begin{cases}(p_1+p_3)(\alpha_1+\alpha_2)(p_2+p_4)+\langle I^A\rangle&l(p_1)=l(p_3)=0 \mbox{ or }l(p_1)>0, l(p_3)>0;\\
(p_1+p_3)(\alpha_1+\alpha_2)(p_2+p_4)-\alpha_1p_2+\langle I^A\rangle& l(p_1)>0 \mbox{ and }l(p_3)=0;\\
(p_1+p_3)(\alpha_1+\alpha_2)(p_2+p_4)-\alpha_2p_4+\langle I^A\rangle& l(p_3)>0 \mbox{ and }l(p_1)=0.\end{cases}$$
\item
Dually, for all paths $p_1'\beta_1p_2'$ and $p_3'\beta_2p_4'$ in $Q^A$ satisfying $p_1'\beta_1p_2', p_3'\beta_2p_4'\notin I^A$, we have 
$(p_1'\beta_1p_2'+\langle I^A\rangle)+(p_3'\beta_2p_4'+\langle I^A\rangle)=$
$$\begin{cases}(p_1'+p_3')(\beta_1+\beta_2')(p_2'+p_4')+\langle I^A\rangle&l(p_2')=l(p_4')=0 \mbox{ or } l(p_2')>0, l(p_4')>0;\\
(p_1'+p_3')(\beta_1+\beta_2')(p_2'+p_4')-p_1'\beta_1+\langle I^A\rangle& l(p_2')>0 \mbox{ and }l(p_4')=0;\\
(p_1'+p_3')(\beta_1+\beta_2')(p_2'+p_4')-p_3'\beta_2+\langle I^A\rangle& l(p_4')>0 \mbox{ and }l(p_2')=0.
\end{cases}$$

\end{enumerate}

\end{enumerate}

\end{Lem}

\begin{proof} $(a)$ Suppose that $i^-\lra s^-$ is another arrow in $Q^A$ with $s^-\in Q^{Sp-}_0$. Then,
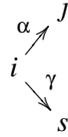
$$\xymatrix@=3mm{
&j\\
i\ar[ur]^{\alpha}\ar[dr]^{\gamma}&\\
&s
}$$
is a sub-quiver of $Q$ with $i$ a special vertex. Since $(Q, I, Sp)$ is a skew-gentle triple, the vertex $i$ cannot be a special vertex. This is a contradiction.  Dually, there is at most one arrow $k^-\lra i^-$ in $Q^A$.

$(b)$ The proof of $(b)$ is similar to that of $(a)$.

$(c)$ Suppose that $p:{}i^-\stackrel{\alpha_1}\lra \cdots \stackrel{\alpha_l}\lra i^-$ is an oriented cycle in $Q^A$ such that $p\notin \langle I^A\rangle$. Note that the zero relations in $I^A$ are paths of length of $2$. Since $i$ is a special vertex, $\alpha_l\alpha_1$ is not a zero relation. 
So, $p^n\notin \langle I^A\rangle$ for all positive integer $n$. Hence $A$ is infinite dimensional. This is a contradiction.

$(d1), (d2)$ and $(e)$ follow from the commutative relations in $I^A$. We assume that $p$ is given by
$$i^-\lra k_1^+\lra k_2^+\lra \cdots\lra k_s^+\lra j^-$$
with all $k_i^{+}\in Q^{Sp+}_0$ for all $i=1, \cdots, s$. Thus, the quiver $Q^A$ contains the following sub-quiver with all squares anti-commutative 
$$
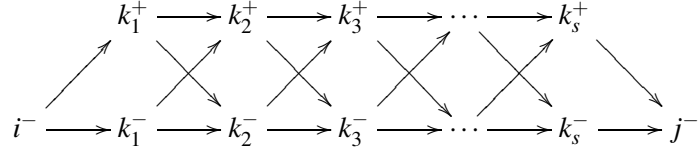
\xymatrix{
&k_1^+\ar[r]\ar[rd]&k_2^+\ar[r]\ar[rd]&k_3^+\ar[r]\ar[rd]&\cdots\ar[r]\ar[rd]&k_s^+\ar[rd]&\\
i^-\ar[r]\ar[ur]&k_1^-\ar[r]\ar[ur]&k_2^-\ar[r]\ar[ur]&k_3^-\ar[r]\ar[ur]&\cdots\ar[r]\ar[ur]&k_s^-\ar[r]&j^-
}$$
Thus, we have
$(-1)^sp=i^-\lra k^-_1\lra \cdots \lra k^-_s\lra j^-\mbox{ modulo }\langle I^A\rangle.$

$(d3)$ For the vertex $i^-$, there is at most one arrow in $Q^A$ with terminal in $Q^{or}_0\cup Q^{Sp+}_0$. Otherwise, there are two arrows going out from the special vertex $i$. It follows that $(Q, I, Sp)$ is not a skew-gentle triple.
 We assume that $t(\alpha_1)$ is an ordinary vertex. Since $(Q, I)$ is a gentle pair, there are at most two arrows out from the vertex $t(\alpha_1)$. Let $\alpha'_2$ and $\alpha''_2$ be arrows starting in $t(\alpha_1)$. It also follows from $(Q, I)$ be a gentle pair that $\alpha_1\alpha_2'\in  I^A$ or $\alpha_1\alpha''_2\in  I^A$.

$(f)$ The assumption that $p_1p_2$ lies in $\langle I^A\rangle$ is equivalent to saying that $\alpha_l\beta_1\in I^A$ and $j$ is an ordinary vertex.

$(g)$ We just prove $(g1)$, and $(g2)$ is dual. 
Since $i$ and $k$ are special vertices, we get $i^-\neq k^-$. Otherwise, $i$ cannot be a special vertex. For the same reason above, the vertex $j$ cannot lie in $Q^{Sp+}_0$. Now, we assume that $j$ is an ordinary vertex. Then, we get two distinct paths $p_1\alpha_1p_2$ and $p_3\alpha_2p_4$ satisfying $s(\alpha_1)=s(\alpha_2)$ and $t(\alpha_1)\neq t(\alpha_2)$.

If $l(p_1)=0, l(p_3)=0$, then $\alpha_1p_2+\alpha_2p_4=(\alpha_1+\alpha_2)(p_3+p_4)$. The conclusion is true. Assume $l(p_1)>0$ and $l(p_3)>0$. Note that the arrows $\alpha_1$ and $\alpha_2$ correspond to arrows in the gentle pair $(Q, I)$. 
We have $p_1\alpha_1\in \langle I^A\rangle, p_3\alpha_2\in \langle I^A\rangle$ or $p_1\alpha_2\in\langle I^A\rangle, p_3\alpha_1\in\langle I^A\rangle$. Since $p_1\alpha_1p_2, p_3\alpha_2p_4\notin I^A$, we get $p_1\alpha_2, p_3\alpha_1\in \langle I^A\rangle$. 
Thus, 
$$(p_1\alpha_1p_2+\langle I^A\rangle)+(p_3\alpha_2p_4+\langle I^A\rangle)=(p_1+p_2)(\alpha_1+\alpha_2)(p_3+p_4)+\langle I^A\rangle.$$
The other cases are similar and straightforward.\end{proof}

For a skew-gentle triple $(Q, I, Sp)$, we divide the set of relations $I$ into two parts. Namely, $I$ is the disjoint union of  
$$I^{or}:{}=\{\cdot\stackrel{\alpha}\lra i\stackrel{\beta}\lra \cdot\in I\mid i\in Q^{or}_0 \} \mbox{ and }I^{Sp}:{}=\{\cdot\stackrel{\alpha}\lra i\stackrel{\beta}\lra \cdot\in I\mid i\in Q^{Sp}_0\}.$$
The next lemma is about the relation between $A(Q, I)$ and $A(Q, I^{or})$.

\begin{Lem}Let $(Q, I, Sp)$ be a skew-gentle triple. Then, $A(Q, I^{or})$ is a gentle algebra, and $A(Q, I)$ is isomorphic to the quotient of $A(Q, I^{or})$ modulo $\langle I\rangle/\langle I^{or}\rangle$.
\end{Lem}

\begin{proof} In fact, the potential special vertices are
$$\xymatrix{
\cdots\ar[r]&i
},
\quad\quad\quad \xymatrix{i&\cdots\ar[l]},\quad\quad\quad
\xymatrix{\cdots\ar[r]^{\alpha}&i\ar[r]^{\beta}&\cdots
}$$
with $\alpha\beta\in I$. Obviously, $(Q, I^{or})$ is still a gentle pair if we can prove that $A(Q, I^{or})$ is finite dimensional. For convenience, we suppose  $Q^{Sp}_0=\{i\}$.

Assume that $A(Q, I^{or})$ is infinite dimensional, and $p:{}\cdots\stackrel{\alpha}\lra i\stackrel{\beta}\lra\cdots$ is a path with infinite length satisfying $p\notin\langle I^{or}\rangle$. Note that $A(Q, I)$ is finite dimensional. Then, we have $p\in \langle I\rangle$. It follows $p\in\langle I^{Sp}\rangle$. So, there is some arrow $\alpha$ in $p$ such that $t(\alpha)=i$ is a special vertex. 
Now, we add a special loop at the vertex $i$, and consider the path $p':{}\cdots\stackrel{\alpha}\lra i\stackrel{\delta
_i}\lra i\stackrel{\beta}\lra \cdots$. Obviously, $p'$ is a path with infinite length, and $p'\notin \langle I\cup \{\delta_i^2\}\rangle$. This is a contradiction, because $A(Q, I\cup\{\delta_i^2\})$ is a finite dimensional algebra. It is immediate to check the isomorphism.
\end{proof}

\medskip

Now, we describe the quivers of $B$ and $C$. 
Recall that, for an idempotent $f$ of $A$, the radical of $fAf$ is $f\rad(A) f$ (see \cite[Proposition 3.2.4]{Drozd1994}). And there is a $k$-linear isomorphism $fAf/f\rad(A)f\lra f(A/\rad(A))f$ which is given by $fxf+f(\rad(A))f\mapsto f(x+\rad(A))f$ with $x\in A$. Thus, the sets of vertices in $Q^B$ (resp. $Q^C$) consists of the ones in $Q^{or}_0\cup Q^{Sp+}_0$ (resp. $Q^{Sp-}_0$).

Note that the arrows of $fAf$ are given by $\rad(fAf)/\rad^2(fAf) (=f\rad(A)f/f\rad(A)f\rad(A)f)$, and there exists a $k$-linear isomorphism
$$f\rad(A)f/f\rad(A)f\rad(A)f\lra f(\rad(A)/\rad(A)f\rad(A))f,$$
which sends $fxf+f\rad(A)f\rad(A)f$ to $f(x+\rad(A)f\rad(A))f$ for all $x\in \rad(A)$. Then, the arrows of $Q^{fAf}$ are determined by the paths in $f\rad(A)f$ satisfying that they do not pass through the vertices corresponding to the idempotent $f$.

Consider the path $p:{}i^-\stackrel{\alpha_1}\lra \cdots\stackrel{\alpha_l}\lra j^-\notin \langle I^A\rangle$ in $Q^A$ with $i^-, j^-\in Q^{Sp-}_0$. If there is some $t(\alpha_{i_0})\in Q^{Sp+}_0\cup Q^{Sp-}_0$, then we have $p+\langle I^A\rangle\in \rad(A)e\rad(A)$ by Lemma \ref{algebra C} $(d1)$. By Lemma \ref{algebra C} $(d3)$, there is a unique $p\notin \langle I^A\rangle$ satisfying $t(\alpha_i)\in Q^{or}_0$ for all $i=1, \cdots, l-1$.
To sum up, the set of arrows $Q^C_1$ consists of two parts:
\begin{itemize}
\item
If ${}^-\alpha^-:{}i^-\lra j^-$ with $i^-, j^-\in Q_0^{Sp-}$ is an arrow in $Q^A_1$, then ${}^-\alpha^-$ lies in $Q^C_1$.
\item
Let $l$ be a positive integer greater than $1$, and $i^-\stackrel{\alpha_1}\lra \cdots\stackrel{\alpha_l}\lra j^-\notin \langle I^A\rangle$ be a path in $Q^A$ with $i^-, j^-\in Q^{Sp-}_0, t(\alpha_i)\in Q^{or}_0$ for all $i=1,\cdots, l-1$. Then there is an arrow $i^-\lra j^-$ in $Q^C_1$.
\end{itemize}

\begin{Koro}\label{hereditary} Let $(Q, I, Sp)$ with $Sp\neq \emptyset$ be a skew-gentle triple. Then, there is an algebra isomorphism $C\lra kQ^C$.
Any indecomposable direct factor of $C$ is a hereditary algebra with the linearly oriented quiver of type $\mathbb{A}_n$ with $n\geqslant 1$.
\end{Koro}

\begin{proof} Let $f:{} kQ^C\lra C$ be a $k$-linear map which sends $i^-\in Q^C_0$ to $i^-+\langle I^A\rangle$, sends an arrow ${}^-\alpha^-:{}i^-\lra j^-\in Q^C_1$ to ${}^-\alpha^-+\langle I^A\rangle$ or $p+\langle I^A\rangle$ where $p:{}i^-\stackrel{\alpha_1}\lra \cdots\stackrel{\alpha_l}\lra j^-\notin \langle I^A\rangle$ is a path in $Q^A$ with $t(\alpha_i)\in Q^{or}_0$ for all $i=1, \cdots, l-1$. It is easy to check that $f$ is surjective and an algebra homomorphism.

Note that the paths of the form
$$i^-\stackrel{{}^-\alpha^-}\lra j^-\stackrel{{}^-\beta^-}\lra k^-, \;\; i^-\stackrel{{}^-\alpha^-}\lra j^-\stackrel{{}^-\beta}\lra k, \;\;i\stackrel{\alpha^-}\lra j^-\stackrel{{}^-\beta^-}\lra k^- $$
are not zero relations in $I^A$. By Lemma \ref{algebra C} $(a)$ and $(d3)$, the image of any combination of paths in $Q^C$ under $f$ does not lie in $I^A$. Therefore, the map $f$ is an algebra isomorphism. 

 Thanks to Lemma \ref{algebra C} $(a), (d3)$, any indecomposable direct factor of $C$ is a Nakayama algebra. And under the help of Lemma \ref{algebra C} $(c)$ that it is a Nakayama algebra of linearly ordered quiver of type $\mathbb{A}_n$ with $n\geqslant 1$.\end{proof}

Next, we turn to the quiver with relations of $B$. Thanks to Lemma \ref{algebra C} $(e)$, the set of arrows $Q^B_1$ is a full subset of $Q^A_1$. Namely, an arrow $\alpha:{}i\lra j$ with $i, j\in Q^B_0$ is in $Q^B_1$ if and only if $\alpha$ lies in $Q^A_1$. For a relation $j\stackrel{\alpha}\lra i\stackrel{\beta}\lra k\in I^{or}$
\begin{itemize}
\item
if $j\in Q^{or}_0$ and $k\in Q^{Sp}_0$, then $\alpha\beta^+\in I^B$,
\item
if $j\in Q^{Sp}_0$ and $k\in Q^{or}_0$, then ${}^+\alpha\beta\in I^B$,
\item
if $j, k\in Q^{or}_0$, then $\alpha\beta\in I^B$.
\end{itemize}

Clearly, the algebra $A(Q^B, I^B)$ is isomorphic to the gentle algebra $A(Q, I^{or})$. The following lemma tells us that $(Q^B, I^B)$ is actually the quiver with relations of $B$.

\begin{Lem} \label{algebra iso}Notation as above. There exists an algebra isomorphism $A(Q^B, I^B)\lra B $, which is given by $p+\langle I^B\rangle\mapsto p+\langle I^A\rangle$ with a path $p\in kQ^B$.
\end{Lem}

\begin{proof} Consider the canonical algebra homomorphism $\pi:{}kQ^B\lra B$ which is given by $p\mapsto p+\langle I^A\rangle$ for all $p\in kQ^B$. The kernel of $\pi$ is indeed the ideal $\langle I^B\rangle$. \end{proof}

At last, we want to describe the left $B$-right $C$-bimodule $M$ and the left $C$-right $B$-bimodule $N$. Set
$$Q^M_1:=\{i\stackrel{\alpha}\lra j\in Q^A_1\mid i\in Q^{or}_0\cup Q^{Sp
+}_0, j\in Q^{Sp-}_{0}\},\; Q^N_1:=\{i\stackrel{\alpha}\lra j\in Q^A_1\mid i\in Q^{Sp-}_0, j\in Q^{or}_0\cup Q^{Sp+}_0\}.$$
Obviously, the set $\{\alpha+\langle I^A\rangle\mid \alpha\in Q^M_1\}$ (resp. $\{\alpha+\langle I^A\rangle\mid \alpha\in Q^N_1\}$) is a generating set of $M$ (resp. $N$) as a left $B$-right $C$-bimodule (resp. a left $C$-right $B$-bimodule).

Now,
we suppose that $C=C_1\times C_2\times\cdots \times C_l$ is a decomposition of the algebra $C$ into a direct product of indecomposable algebras $C_1, \cdots, C_l$.
By Corollary \ref{hereditary}, the algebra $C_p$ is a hereditary algebra with linearly oriented quiver for $p=1, \cdots, l$. In this way, the set $Q^{Sp-}_0$ is a disjoint union of $Q^{Sp-}_{01}, \cdots, Q^{Sp-}_{0l}$ where the simple modules of $C_p$ correspond to the vertices in $Q^{Sp-}_{0p}$ for all $p=1, \cdots, l$. Hence, we have $e=e_1+\cdots+e_l$ with $e_p=\sum_{q\in Q^{Sp-}_{0p}}e_q$ for all $p=1, \cdots, l$. We  
set 
$$Q^{M}_{1p}:=\{i\stackrel{\alpha}\lra j\in Q^A_1\mid i\in Q^{or}_0\cup Q^{Sp
+}_0, j\in Q^{Sp-}_{0p}\},\;Q^N_{1q}:{}=\{i\stackrel{\alpha}\lra j\in Q^A_1\mid i\in Q^{Sp-}_{0q} , j\in Q^{or}_0\cup Q^{Sp
+}_0\}$$
for all $p, q=1, \cdots, l$. It follows that $Q^M_1$ (resp. $Q^N_1$) is a disjoint union of $Q^M_{11}, \cdots, Q^M_{1l}$ (resp. $Q^N_{11}, \cdots, Q^N_{1l}$).

Let $p_1\alpha_1p_2$ and $p_3\alpha_2 p_4$ be two paths in $Q^A$ with $\alpha_1, \alpha_2\in Q^M_1$. If $s(\alpha_1)\neq s(\alpha_2)$ and $t(\alpha_1)\neq t(\alpha_2)$, then $p_1\alpha_1p_2+p_3\alpha_2p_4=(p_1+p_3)(\alpha_1+\alpha_2)(p_2+p_4)$. And it is impossible that $s(\alpha_1)\neq s(\alpha_2)$ and $t(\alpha_1)=t(\alpha_2)$ for $t(\alpha_1)$ being a special vertex. Then, by Lemma \ref{algebra C} $(g)$, we have $$\sum_{\alpha\in Q^M_{1p}}(1-e)Ae_{s(\alpha)}\alpha e_{t(\alpha)}Ae=B(\sum_{\alpha\in Q^M_{1p}}\alpha) C.$$
Dually, we get
$$\sum_{\alpha\in Q^N_{1q}}eAe_{s(\alpha)}\alpha e_{t(\alpha)}A(1-e)=C(\sum_{\alpha\in Q^N_{1q}}\alpha)B.$$
Consequently, we have 
$$M=(1-e)Ae=\sum_{\alpha\in Q^M_1}(1-e)Ae_{s(\alpha)}\alpha e_{t(\alpha)}Ae=\sum_{p=1, \cdots, l}\sum_{\alpha\in Q^M_{1p}}(1-e)Ae_{s(\alpha)}\alpha e_{t(\alpha)}Ae=\sum_{p=1, \cdots, l}B(\sum_{\alpha\in Q^M_{1p}}\alpha) C,$$
$$N=eA(1-e)=\sum_{\alpha\in Q^N_1}eAe_{s(\alpha)}\alpha e_{t(\alpha)}A(1-e)=\sum_{q=1, \cdots, l}\sum_{\alpha\in Q^N_{1q}}eAe_{s(\alpha)}\alpha e_{t(\alpha)}A(1-e)=\sum_{q=1, \cdots, l}C(\sum_{\alpha\in Q^N_{1q}}\alpha)B.$$

By Corollary \ref{algebra C} and Lemma \ref{algebra iso}, the algebras $B$ and $C$ are isomorphic to the quiver algebras $A(Q^B, I^B)$ and $A(Q^C, \emptyset)$, respectively. Then, a path $p_1\alpha p_2$ with $p_1\in Q^B, p_2\in Q^C$ and $\alpha\in Q^M_1$ is actually a path $p_1'\alpha p_2'$ in $Q^A$ satisfying $s(p_1'), t(p_2')\in Q^{or}_0\cup Q^{Sp+}_0$, where $p_1'$ and $p_2'$ are the images of $p_1$ and $p_2$ under the algebra isomorphisms, respectively.

\begin{Lem} \label{bimodules}For all arrows in $Q^M_{1p}$ and $Q^N_{1q}$ with $p, q=1, \cdots, l$, we have
\begin{enumerate}[label=(\alph*)]
\item
Let $\alpha$ and $\beta$ be arrows in $Q^M_{1p}$ and $Q^N_{1q}$, respectively. 
Assume that $p_1\alpha p_2$ is a path with $p_1\in Q^B, p_2\in Q^C$ satisfying $p_1\alpha p_2\notin \langle I^A\rangle$. Then, there exists an arrow $\alpha'\in Q^M_{1p}$ and an element $\widetilde{p}_1\in kQ^B$ such that $p_1\alpha p_2-\widetilde{p}_1\alpha'\in\langle I^A\rangle$. 

Dually, suppose that $p_3\beta p_4$ is a path such that $p_3\beta p_4\notin \langle I^A\rangle$ with $p_3\in Q^C, p_4\in Q^B$. Then, there is an arrow $\beta'\in Q^N_{1q}$ and an element $\widetilde{p}_4\in kQ^B$ satisfying $p_3\beta p_4-\beta'\widetilde{p}_4'\in\langle I^A\rangle$.

\item
For all $1\leqslant p\neq q\leqslant l$, we have
$$B(\sum_{\alpha\in Q^M_{1p}}\alpha)C\bigcap B(\sum_{\alpha\in Q^M_{1q}}\alpha)C=\{0\}, \;C(\sum_{\alpha\in Q^N_{1p}}\alpha)B\bigcap C(\sum_{\alpha\in Q^N_{1q}}\alpha)B=\{0\}.$$

\end{enumerate}
\end{Lem}

\begin{proof} $(a)$ We just prove the first statement. The second one is dual. Assume that $\alpha:{}i\lra j^-$ is an arrow with $i\in Q^{or}_0\cup Q^{Sp+}_0$ and $j^-\in Q^{Sp-}_0$. Note that $C_p$ is the indecomposable algebra corresponding to the set $Q^{M}_{1p}$. By Corollary \ref{hereditary}, the algebra $C_p$ is a hereditary algebra with linearly oriented quiver of type $\mathbb{A}_n$ for $n\geqslant 1$. If there is not arrow $j^-\lra k^-$ in $eAe$ with $k^-\in Q^{Sp-}_0$, then the path $p_1\alpha p_2$ lies in $\langle I^A\rangle$, where the length of $p_2$ is greater than $0$, because the bimodule $B\alpha C$ is a simple right $C$-module. 
Since $p_1\alpha p_2\notin \langle I^A\rangle$, the path $p_1\alpha p_2$ is actually $p_1\alpha e_{t(\alpha)}$.

Now, we assume that there is an arrow $\alpha_1:{}j^-\lra k^-$ in $eAe$ with $k^-\in Q^{Sp-}_0$, and the length of $p_2$ is greater than $0$. Then, it follows from the commutativity relations in $I^A$ that there exist elements $p'_1\in Q^B$ and $p_2'\in Q^C$ such that $p_1\alpha p_2+p_1'\alpha_1 p_2'\in \langle I^A\rangle$. Note that the length of $p_2'$ is indeed $l(p_2)-1$. Then, by induction on $l(p_2)$, we finally get an arrow $\alpha'\in Q^M_{1q}$ and an element $\widetilde{p}_1\in kQ^B$ such that $p_1\alpha p_2-\widetilde{p}_1\alpha'\in\langle I^A\rangle$.

$(b)$ Let $x\in B(\sum_{\alpha\in Q^M_{1p}}\alpha)C\bigcap B(\sum_{\alpha\in Q^M_{1q}}\alpha)C$. Note that the set $B(\sum_{\alpha\in Q^M_{1p}}\alpha )C$ is actually $\sum_{\alpha\in Q^M_{1p}}B\alpha C$ for all $p=1, \cdots, l$. By $(a)$, we set $x=\sum_{\alpha\in Q^{M}_{1p}}k_\alpha p_{\alpha}\alpha+ \langle I^A\rangle$ and $x=\sum_{\beta\in Q^{M}_{1q}}k_\beta p_{\beta}\beta+ \langle I^A\rangle$ with $k_{\alpha}, k_\beta\in k$ and $p_\alpha, p_\beta\in Q^B$ for all $\alpha\in Q^M_{1p}$ and $\beta\in Q^M_{1q}$. Here, we assume that both the sets $\{p_{\alpha}\alpha+\langle I^A\rangle\mid \alpha\in Q^M_{1p}\}$ and $\{p_{\beta}\beta+\langle I^A\rangle\mid \beta\in Q^M_{1q}\}$ are $k$-linearly independent.

Since $p$ and $q$ are different positive integers, the set $\{p_{\alpha}\alpha+\langle I^A\rangle, p_\beta\beta+\langle I^A\rangle\mid \alpha\in Q^M_{1p}, \beta\in Q^M_{1q}\}$ is $k$-linearly independent, because the paths $p_{\alpha}\alpha$ and $p_{\beta}\beta$ have different terminals. 
It follows from $\sum_{\alpha\in Q^{M}_{1p}}k_\alpha p_{\alpha}\alpha-\sum_{\beta\in Q^{M}_{1q}}k_\beta p_{\beta}\beta\in\langle I^A\rangle$ that $k_{\alpha}=0=k_\beta$ for all $\alpha\in Q^M_{1p}$ and $\beta\in Q^M_{1q}$. The second statement is dual. 
\end{proof}

Similar with Lemma \ref{bimodules} $(b)$, we get 
$$B(\sum_{\alpha\in Q^M_{1p}}\alpha)C\bigcap (\sum_{q\neq p} B(\sum_{\alpha\in Q^M_{1q}}\alpha)C)=\{0\}, \;C(\sum_{\alpha\in Q^N_{1p}}\alpha)B\bigcap (\sum_{q\neq p}C(\sum_{\alpha\in Q^N_{1q}}\alpha)B)=\{0\}$$
for all $p=1, \cdots, l$. It follows that 
$M=\bigoplus^l_{p=1}B(\sum_{\alpha\in Q^M_{1p}}\alpha)C$ and $N=\bigoplus^l_{q=1}C(\sum_{\alpha\in Q^N_{1q}}\alpha)B$ as a left $B$-right $C$-bimodule and a left $C$-right $B$-bimodule, respectively. 
In the following, we will describe the bimodules $M$ and $N$ explicitly. 
For all $n\geqslant 1$, let
$$1^-\stackrel{{}^-\alpha_1^-}\lra 2^-\stackrel{{}^-\alpha^-_2}\lra 3^-\stackrel{{}^-\alpha_3^-}\lra\cdots \stackrel{{}^-\alpha_{n-1}^-}\lra n^-$$
be some direct factor $C_{p}$. Note that the arrows in the quiver of $C_p$ correspond to arrows or paths in $Q^A$. Then, we should consider the following two cases. 
The first case is that there are no arrows from the ordinary vertex to the special vertices $1^+$ and $1^-$. Then the following is a sub-quiver of $Q^A$ with all squares anti-commutative. 
\begin{equation}
\xymatrix@=5mm{1^+\ar[r]^{}\ar[rdd]_(.3){}&2^+\ar[r]^{}\ar[rdd]_(.3){}&3^+\ar[r]^{}\ar[rdd]_(.3){}&\cdots\ar[r]^{}\ar[rdd]_(.3){}&s^+\ar[rd]&&&&(n-t)^+\ar[r]\ar[rdd]&\cdots\ar[r]\ar[rdd]&n^+\ar[rd]&\\
&&&&&\cdot\ar[r]&\cdots\ar[r]&\cdot\ar[ur]\ar[rd]&&&&\cdot\\
1^-\ar[r]^{{}^-\alpha_1^-}\ar[uur]^(.3){}&2^-\ar[r]^{{}^-\alpha_2^-}\ar[uur]^(.3){}&3^-\ar[r]^{{}^-\alpha_3^-}\ar[uur]&\cdots\ar[r]^{{}^-\alpha_{s-1}^-}\ar[uur]&s^-\ar[ur]&&&&(n-t)^-\ar[r]\ar[uur]&\cdots\ar[r]\ar[ruu]&n^-\ar[ur]&}\end{equation}

In this situation, the left $B$-right $C$-bimodule $B(\sum_{\alpha\in Q^M_{1p}}\alpha)C$ is given by $(\sum_{j=1}^{n-1}e_{j^+})A(\sum_{\alpha\in Q^M_{1p}}\alpha)Ae$ which is actually $(\sum_{\alpha\in Q^M_{1p}}\alpha)Ae$ for the commutativity relations in $I^A$. Thus, the $k$-vector space $(\sum_{\alpha\in Q^M_{1p}}\alpha)Ae$ is a left $B$-right $C$-bimodule. 
Note that $C_p$ is a hereditary algebra of the linearly oriented quiver of type $\mathbb{A}_n$, and the set of relations $I^A$ consists of zero relations and commutative relations. Then, the map
$$f_p:{}(\sum_{\alpha\in Q^M_{1p}}e_{t(\alpha)})Ae\lra (\sum_{\alpha\in Q^M_{1p}}\alpha)Ae$$ which sends $\sum_{\alpha\in Q^M_{1p}}k_{\alpha}e_{t(\alpha)}p_\alpha+\langle I^A\rangle$ to $\sum_{\alpha\in Q^M_{1p}}k_{\alpha}\alpha p_{\alpha}+\langle I^A\rangle$ is a $k$-linear isomorphism with $k_{\alpha}\in k$ and paths $p_{\alpha}\in Q^C$ for all $\alpha\in Q^M_{1p}$.

In fact, the $k$-vector space $(\sum_{\alpha\in Q^M_{1p}}e_{t(\alpha)})Ae$ is a left $B$-right $C$-bimodule. The right $C$-module structure is the multiplication in $A$, and the left $B$-module structure is given by multiplication by $(\sum_{\alpha\in Q^M_{1p}}\alpha)$. 
Set $b\in B$ and $x\in (\sum_{\alpha\in Q^M_{1p}}e_{t(\alpha)})Ae$. 
For the multiplication $b(\sum_{\alpha\in Q^M_{1p}}\alpha)x$, there exists $x'\in kQ^C$ such that $b(\sum_{\alpha\in Q^M_{1p}}\alpha)x-(\sum_{\alpha\in Q^M_{1p}}\alpha)x'x\in \langle I^A\rangle$ by the commutative relations in $I^A$. 
Then, the left $B$-module structure $b\cdot x$ is indeed given by $f^{-1}_p((\sum_{\alpha\in Q^M_{1p}}\alpha)x'x)$. And it is straightforward that $f_p$ is a left $B$-right $C$ isomorphism.

The second case is that there are arrows $i\lra 1^{\pm}$ with $i$ an ordinary vertex. The corresponding sub-quiver of $Q^A$ is presented as follows.
\begin{equation}
\xymatrix@=5mm{&1^+\ar[r]^{}\ar[rdd]_(.3){}&2^+\ar[r]^{}\ar[rdd]_(.3){}&3^+\ar[r]^{}\ar[rdd]_(.3){}&\cdots\ar[r]^{}\ar[rdd]_(.3){}&s^+\ar[rd]&&&&(n-t)^+\ar[r]\ar[rdd]&\cdots\ar[r]\ar[rdd]&n^+\ar[rd]&\\
i\ar[ur]\ar[rd]_{\alpha^-}&&&&&&\cdot\ar[r]&\cdots\ar[r]&\cdot\ar[ur]\ar[rd]&&&&\cdot\\
&1^-\ar[r]^{{}^-\alpha^-_1}\ar[uur]^(.3){}&2^-\ar[r]^{{}^-\alpha_2^-}\ar[uur]^(.3){}&3^-\ar[r]^{{}^-\alpha_3^-}\ar[uur]&\cdots\ar[r]^{{}^-\alpha_{s-1}^-}\ar[uur]&s^-\ar[ur]&&&&(n-t)^-\ar[r]\ar[uur]&\cdots\ar[r]\ar[ruu]&n^-\ar[ur]&}\end{equation}

Let $p_1(\sum_{\alpha\in Q^M_{1p}}\alpha)p'_1+\langle I^A\rangle$ be an element in $B(\sum_{\alpha\in Q^M_{1p}}\alpha)C$ with $p_1\in Q^B$ and $p_1'\in Q^C$. If the path $p_1$ passes through the vertex $i$, then there exists a subpath of $p_1$, which is denoted by $p_2$, and a path $p_2'\in Q^C$ such that $p_1(\sum_{\alpha\in Q^M_{1p}}\alpha)p_1'-p_2\alpha^-p_2'\in\langle I^A\rangle$. If the path $p_1$ does not pass through the vertex $i$, then the discussion is similar with the first case. In this way, the $k$-vector space $B(\sum_{\alpha\in Q^M_{1p}}\alpha)C$ can be written as $(1-e)Ae_i\otimes_ke_{1^-}Ae\oplus (\sum_{\alpha\in Q^M_{1p}}e_{t(\alpha)})Ae$ which is actually a $\mathbb{N}$-graded left $B$-right $C$-bimodule. And the grading depends on the length of $p_2$. More precisely, the $0$-th degree is the $k$-vector space $\sum_{\alpha\in Q^M_{1p}}e_{t(\alpha)}Ae$, and the $m$-th degree is a $k$-vector space generated by $p_2\alpha^-p_2'+\langle I^A\rangle$ with $p_2\alpha^-p_2'\notin \langle I^A\rangle$ and $l(p_2)=m$ for all $p_2\in Q^B, p_2'\in Q^C$ and $m\geqslant 1$.

Dually, a similar discussion still works for $C(\sum_{\alpha\in Q^N_{1q}}\alpha)B$. Thus, we get the following corollary.

\begin{Koro}\label{projective} Let $(Q, I, Sp)$ be a skew-gentle triple with $Sp\neq \emptyset$. Then,

\begin{enumerate}[label=(\alph*)]
\item
The left $B$-right $C$-bimodule $M$ is a projective right $C$-module. 
\item

The left $C$-right $B$-bimodule $N$ is a projective left $C$-module. 

\item

The $B$-$B$-bimodule homomorphism 
$f:{} M\otimes_{C}N\lra B$, which is given by the multiplication in $A$ is injective.
The image is 
$$\sum_{p=1, \cdots, l}B(\sum_{\alpha\in Q^M_{1p}}\alpha)C(\sum_{\beta\in Q^N_{1p}}\beta)B.$$
which
is an ideal of $B$, and a generating set is given as follows:
$$ \alpha^+{}^+\beta, \alpha^+{}^+\beta^+, {}^+\alpha^+{}^+\beta, {}^+\alpha^+{}^+\beta^+$$
for a relation $j\stackrel{\alpha}\lra i\stackrel{\beta}\lra k\in I$ with $i$ a special vertex. Namely, the quotient algebra $B/f(M\otimes_CN)$ is isomorphic to the quiver algebra $A(Q, I)$.

\end{enumerate}

\end{Koro}

\begin{proof} $(a)$ By Lemma \ref{bimodules} and the discussion above, $M$ has direct summands of the form $(\sum_{\alpha\in Q^M_{1p}}e_{t(\alpha)})Ae$ or $(1-e)Ae_i\otimes_ke_{1^-}Ae\oplus (\sum_{\alpha\in Q^M_{1p}}e_{t(\alpha)}Ae)$ as left $B$-right $C$-bimodules. Obviously, they are projective right $C$-modules for $\{1^-, t(\alpha)\mid \alpha\in Q^M_{1p}\}$ being a subset of $Q^{Sp-}_0$. Similarly, we get $(b)$.

$(c)$ By Lemma \ref{bimodules} $(a)$, the left $B$-right $C$-bimodule $B(\sum_{\alpha\in Q^M_{1p}}\alpha)C$ can be written as $B(\sum_{\alpha\in Q^M_{1p}}\alpha)$. Dually, the left $C$-right $B$-bimodule $C(\sum_{\beta\in Q^N_{1p}}\beta)B$ is $(\sum_{\beta\in Q^N_{1p}}\beta)B$.

Now, we write $M_p=B(\sum_{\alpha\in Q^{M}_{1p}}\alpha)$ and $N_p=(\sum_{\beta\in Q^N_{1p}}\beta)B$ for short. By Lemma \ref{bimodules} $(b)$, we have
$M=\bigoplus^l_{p=1}M_p$ and $N=\bigoplus^l_{p=1}N_p$. Note that $C$ is the direct product $C_1\times C_2\times \cdots \times C_l$. And $M_p$ and $N_p$ are a right $C_p$-module and a left $C_p$-module, respectively.  Thus, we have $M_p\otimes_CN_q=0$ for all $1\leqslant p\neq q\leqslant l$. It follows $M\otimes_CN=\bigoplus_{p=1}^lM_p\otimes_{C_p}N_p$. So, the  image of $f$ is
$$\sum_{p=1, \cdots, l}B(\sum_{\alpha\in Q^M_{1p}}\alpha)C(\sum_{\beta\in Q^N_{1p}}\beta)B,$$
 and the generating set of $f(M\otimes_CN)$ as a $B$-$B$-bimodule is 
$\{ \alpha^+{}^+\beta, \alpha^+{}^+\beta^+, {}^+\alpha^+{}^+\beta, {}^+\alpha^+{}^+\beta^+\mid j\stackrel{\alpha}\lra i\stackrel{\beta}\lra k\in I^{Sp}\}.$

What is left to show is $f_p:{}M_p\otimes_{C_p}N_p\lra B$ is injective for all $p=1, \cdots, l$. 
By Lemma \ref{algebra C} $(e)$, it suffices to prove that any path $i\stackrel{\alpha_1}\lra \cdots \stackrel{\alpha_n}\lra j$ does not lie in $\langle I^A\rangle$
with $i, j\in Q^{or}_0\cup Q^{Sp+}_0$ and $t(\alpha_i)\in Q^{Sp+}_0$ for all $i=1, \cdots, n-1$. This is obvious by the zero relations in $I^A$.\end{proof}

\begin{Prop}\label{skew-gentle sd} Let $(Q, I, Sp)$ be a skew-gentle triple with $Sp\neq\emptyset$. Then
\begin{enumerate}[label=(\alph*)]
\item
The set $A E_{22}A$ is a stratifying ideal of $A$.

\item
  $\pd_{A^e}(A E_{22}A)<\infty$.
  
  \item
  The two algebras $A/A E_{22}A$ and $A(Q, I)$ are isomorphic.

\end{enumerate}

\end{Prop}

\begin{proof} $(a)$ It is immediate from Corollary \ref{projective} and Lemma \ref{KN09}.

$(b)$ This follows essentially from the proof of \cite[Proposition 3.3]{Koenig2009}. For convenience of the readers, we include the proof. By $(a)$, $A E_{22}A$ is a stratifying ideal of $A$. Thus, $A E_{22}A$ is isomorphic to $A E_{22}\otimes_{C}E_{22}A$ as $A$-$A$-bimodules. Let $X$ be a left $A^e$-module. We have
$$\begin{array}{ll}
\Ext^n_{A^e}(A E_{22}A, X)&\cong \Ext^n_{A^e}(A E_{22}\otimes_{C}E_{22}A, X)\\
&\cong \Ext^n_{A\otimes_kC^{\opp}}(A E_{22}, \Hom_{A^{\opp}}(E_{22}A, X))\\
&\cong \Ext^n_{A\otimes_kC^{\opp}}(A E_{22}\otimes_{C}C, XE_{22})\\
&\cong \Ext^n_{C^{e}}(C, E_{22}XE_{22}).
\end{array}
$$
By \cite[Section 1.5]{Happel1989d}, we have $\gldim (C)=\pd_{C^e}C$. So, it follows from Corollary \ref{hereditary} that the projective dimension of $A E_{22}A$ as a left $A^e$-module is not greater than $1$. 

$(c)$ By Corollary \ref{projective} $(c)$.
\end{proof}

{\em Proof of Theorem \ref{main result}:} It is a consequence of Proposition \ref{skew-gentle sd} and Corollary \ref{hereditary}. $\square$

\medskip

 If $Sp=\emptyset$, then the skew-gentle algebra is just a gentle algebra. Now, we assume $Sp\neq\emptyset$.
By Theorem \ref{main result}, we get a recollement (\ref{recollement}) with $\gldim(C)\leqslant 1$ and $A/A E_{22}A\cong A(Q, I)$. It follows from Proposition \ref{1} that (\ref{recollement}) can be restricted to $\Db{\rm mod}$.

\medskip

{\em Proof of Corollary \ref{coro}}
 $(a)$ Recall that the singularity category $\Dsg{A}$ is the Verdier quotient of $\Db{A}$ with respect to $\Kb{\pmodcat{A}}$.
 Since (\ref{recollement}) can be restricted to $\Db{\rm mod}$ and $\Kb{\rm proj}$, by \cite[Lemma 1.2]{Orlov2004}, the adjoint pair $(i^*, i_*)$ induces an adjoint pair $(\overline{i^*}, \overline{i_*})$ between the singularity categories $\Dsg{A/A E_{22}A}$ and $\Dsg{A}$. It follows from Theorem \ref{main result} $(c)$ that $(\overline{i^*}, \overline{i_*})$ is a triangle equivalence.

$(b)$ By \cite{Geis2005} and \cite[Theorem 3.3]{ChenLu2017}, all skew-gentle algebras are Gorenstein. 
It follows from \cite{Dalezios2021} that $A(Q, I, Sp)$ and $A(Q, I)$ are singularly equivalent of Morita type with level.

$(c)$ Assume that $A(Q, I, Sp)$ is indecomposable self-injective. Then, by \cite[Theorem 3.6]{CK2017}, the skew-gentle algebra is derived equivalent to $A(Q, I)$ or $C$. This is a contradiction, because derived equivalences preserves the number of isomorphism classes of simple modules. So, we get $Sp=\emptyset$, and $A(Q, I)$ is an indecomposable selfinjective gentle algebra. 

Assume that $A(Q, I)$ is an indecomposable selfinjective gentle algebra. It follows 
$\gldim(A(Q, I))=0$ or $\infty$. If $\gldim(A(Q, I))=0$, then $A(Q, I)$ is just $k$ for $A(Q, I)$ being indecomposable. Now, we suppose that
$\gldim(A(Q, I))=\infty$. Note that the singularity category $\Dsg{A(Q, I)}$ is not trivial, and is given by the Gorenstein projective (not projective) modules. By \cite[Theorem 2.5]{Kalck2015}, $A(Q, I)$ has a full cycle of relations, that is there exists a oriented cycle $\alpha_0\alpha_1\cdots\alpha_n$ such that $s(\alpha_0)=t(\alpha_n)$ and $\alpha_i\alpha_{i+1}, \alpha_n\alpha_0\in I$ for all $i=0, \cdots, n-1$. Since $A(Q, I)$ is selfinjective, any indecomposable projective module over $A(Q, I)$ is injective. So the socle of any indecomposable projective module is simple. Thus, there are no other arrows starting (or ending) at $s(\alpha_i)$ for all $i=0, \cdots, n$. Hence, $A(Q, I)$ is a selfinjective Nakayama algebra with radical square zero.

$(d)$ By \cite[Corollary 6.8]{AKLY2017} (or \cite[Corollary 1.3]{ChenXi2016a}), the K-group of $A(Q, I, Sp)$ is the direct sum of K-groups of $A(Q, I)$ and hereditary algebras of type $\mathcal{A}_n$ for all $n\geqslant 1$. Note that the derived categories of hereditary algebras of type $\mathcal{A}_n$ admit stratifications which are given by sink or source. Then, we use \cite[Corollary 6.8]{AKLY2017} again, and we get that the K-groups of a hereditary algebra of type $\mathcal{A}_n$ are isomorphic to the direct sum of $n$ copies of
the K-groups of the ground field $k$.

\medskip

{\em Proof of Corollary \ref{coro2}}
$(a)$ By Proposition \ref{skew-gentle sd} $(b)$, we have $i^!(A)\in\Kb{\pmodcat{A/A E_{22}A}}$. By Proposition \ref{1}, the recollement can be restricted to $\Kb{\rm proj}$. Due to Proposition \ref{2}, the skew-gentle algebra $A(Q, I, Sp)$ is Gorenstein.

 $(b)$ By $(a)$, skew-gentle algebras are Gorenstein. It is equivalent to saying that $\Omega^d(X)$ is Gorenstein projective for all $X\in \modcat{A(Q, I, Sp)}$, where $d$ is the injective dimension of left regular module. And the projective dimension of a Gorenstein projective module is $0$ or $\infty$. Hence, the finitistic dimension conjecture is true for all skew-gentle algebras.

 $(c)$ By $(a)$ and Proposition \ref{2}, the recollement {\ref{2}} can be restricted to $\Kb{{\rm proj}}$ and $\Kb{{\rm inj}}$. Then, the recollement (\ref{2}) can be extended to a ladder of height $\infty$. It follows from \cite[Theorem 1.2]{CHQW2023} that Auslander and Reiten's conjecture holds for all skew-gentle algebras.

$(d)$ By \cite{ChenLiWang2020}, Keller's conjecture is invariant under singular equivalences of Morita type with level, and it holds for all gentle algebras. Due to Corollary \ref{coro2} $(b)$, Keller's conjecture holds for all skew-gentle algebras. $\square$

\medskip

{\bf Acknowledgements.} The author wants to express her gratitude to Steffen Koenig for so many valuable suggestions, and to all colleagues in Stuttgart for hospitality and discussions.

\footnotesize{
\providecommand{\bysame}{\leavevmode\hbox to3em{\hrulefill}\thinspace}
\providecommand{\noopsort}[1]{}
\providecommand{\mr}[1]{\href{http://www.ams.org/mathscinet-getitem?mr=#1}{MR~#1}}
\providecommand{\zbl}[1]{\href{http://www.zentralblatt-math.org/zmath/en/search/?q=an:#1}{Zbl~#1}}
\providecommand{\jfm}[1]{\href{http://www.emis.de/cgi-bin/JFM-item?#1}{JFM~#1}}
\providecommand{\arxiv}[1]{\href{http://www.arxiv.org/abs/#1}{arXiv~#1}}
\providecommand{\doi}[1]{\url{https://doi.org/#1}}
\providecommand{\MR}{\relax\ifhmode\unskip\space\fi MR }
\providecommand{\MRhref}[2]{%
  \href{http://www.ams.org/mathscinet-getitem?mr=#1}{#2}
}
\providecommand{\href}[2]{#2}
}



\bigskip

\noindent
Yiping Chen, School of Mathematics and Statistics, Wuhan University, Wuhan, Hubei, 430072,  China

\medskip
\noindent
Hubei Key Laboratory of Computational Science (Wuhan University), Wuhan, Hubei, 430072, China

{\tt Email: ypchen@whu.edu.cn}

\end{document}